\documentclass[a4paper,12pt]{article}

\usepackage[utf8]{inputenc}
\usepackage[english,russian]{babel}
\usepackage{amsmath, amssymb, amsthm}
\usepackage{enumerate}
\usepackage[colorlinks=true,linkcolor=blue,citecolor=blue]{hyperref}

\usepackage{tikz}
\usepackage{mdframed}
\usepackage{pgfplots}
\pgfplotsset{width=7cm,compat=1.10}
\pgfplotsset{
  /pgfplots/colormap={pink}{%
    color(0cm) = (blue);
    color(1cm) = (cyan!50!blue);
    color(2cm) = (cyan!50);
    color(3cm) = (cyan) }
}
\usetikzlibrary{calc,arrows,decorations.pathreplacing,fadings,3d,positioning}

\title{О слабой равномерной диофантовой экспоненте вещественного числа
       \thanks{Исследование выполнено за счёт гранта Российского научного фонда № 25-11-00112, https://rscf.ru/project/25-11-00112/}}
\author{О.\,Н.\,Герман}
\date{}

\theoremstyle{definition}
\newtheorem{definition}{Определение}

\newtheorem*{notation*}{Обозначение}

\theoremstyle{remark}

\newtheorem*{remark*}{Замечание}

\theoremstyle{plain}
\newtheorem{theorem}{Теорема}
\newtheorem{lemma}{Лемма}

\newtheorem*{corollary*}{Следствие}

\DeclareMathOperator{\conv}{conv}

\renewcommand{\phi}{\varphi}

\renewcommand{\vec}[1]{\mathbf{#1}}
\renewcommand{\geq}{\geqslant}
\renewcommand{\leq}{\leqslant}

\newcommand{\e}{\varepsilon}

\newcommand{\R}{\mathbb{R}}
\newcommand{\Z}{\mathbb{Z}}
\newcommand{\Q}{\mathbb{Q}}
\newcommand{\N}{\mathbb{N}}

\newcommand{\La}{\Lambda}

\newcommand{\cB}{\mathcal{B}}

\newcommand{\cH}{\mathcal{H}}

\newcommand{\cK}{\mathcal{K}}

\newcommand{\cP}{\mathcal{P}}
\newcommand{\cQ}{\mathcal{Q}}

\begin{document}

\maketitle

\begin{abstract}
  В данной работе мы вводим понятие слабой равномерной диофантовой экспоненты вещественного числа и получаем полное описание спектра её значений.
\end{abstract}

\section{Введение}


В 1842 году в работе \cite{dirichlet} Дирихле опубликовал свою знаменитую теорему, положившую начало теории диофантовых приближений. Для одного вещественного числа она имеет совсем простой вид. Пусть $\theta\in\R$. Рассмотрим систему неравенств 
\begin{equation} \label{eq:dirichlet_one_number}
\begin{cases}
  |x|\leq t \\
  |\theta x-y|\leq t^{-\gamma}
\end{cases}.
\end{equation}
Теорема Дирихле говорит, что эта система имеет ненулевое целочисленное решение (относительно $x,y$) при $\gamma=1$ и произвольном $t\geq1$.

\begin{definition} \label{def:belpha_one_number}
  Супремум вещественных чисел $\gamma$, для которых существует сколь угодно большое $t$, такое что (соотв. для которых при любом достаточно большом $t$) система неравенств \eqref{eq:dirichlet_one_number} имеет ненулевое решение $(x,y)\in\Z^2$, называется \emph{регулярной} (соотв. \emph{равномерной}) \emph{диофантовой экспонентой} числа $\theta$ и обозначается $\omega(\theta)$ (соотв. $\hat\omega(\theta)$).
\end{definition}

Из теоремы Дирихле мгновенно следуют <<тривиальные>> оценки
\[
  \omega(\theta)\geq\hat\omega(\theta)\geq 1.
\]
Для рациональных $\theta$ очевидно, что $\omega(\theta)=\hat\omega(\theta)=\infty$. Для иррациональных $\theta$ экспонента $\omega(\theta)$ может принимать все значения из отрезка $[1,\infty]$ (см. параграф \ref{sec:continued_fractions_and_lattices}). Тогда как экспонента $\hat\omega(\theta)$ для иррациональных $\theta$ оказывается <<вырожденной>>, ибо принимает лишь значение $1$. Причина этого в том, что при каждом $k\in\N$ параллелограмм
\[
  \Big\{(x,y)\in\R^2 \,\Big|\, |x\theta-y|\leq|q_k\theta-p_k|,\ |x|\leq q_{k+1} \Big\},
\]
где $p_k/q_k$, $p_{k+1}/q_{k+1}$ --- последовательные подходящие дроби числа $\theta$, не содержит целочисленных точек, отличных от $\vec 0,\pm(q_k,p_k),\pm(q_{k+1},p_{k+1})$, а площадь треугольника с вершинами в точках $\vec 0,(q_k,p_k),(q_{k+1},p_{k+1})$ в точности равна $1/2$, ибо $|p_kq_{k+1}-q_kp_{k+1}|=1$. Подробности можно найти, например, в обзоре \cite{german_UMN_2023}.

Понятие диофантовой экспоненты естественным образом возникает также в контексте изучения решёток. Исследование некоторых видов равномерных диофантовых экспонент решёток (см. \cite{german_comm_math_2023,german_izv_2025}) привело к более <<тонкому>> варианту равномерной диофантовой экспоненты вещественного числа, спектр которой оказывается нетривиальным. Изучению этой экспоненты и посвящена данная статья. 

В параграфе \ref{sec:weak_uniform_exponent_of_a_number} мы даём определение \emph{слабой равномерной диофантовой экспоненты} вещественного числа и формулируем основной результат статьи, описывающий спектр значений этой экспоненты. Параграф \ref{sec:lattice_exponents} посвящён диофантовым экспонентам решёток. Мы даём соответствующие определения и приводим известные факты об устройстве спектров этих экспонент. В параграфе \ref{sec:continued_fractions_and_lattices} мы описываем связь между диофантовыми экспонентами числа и диофантовыми экспонентами решёток в $\R^2$, после чего переформулируем основной результат работы в терминах диофантовых экспонент решёток. В параграфе \ref{sec:relative_and_hyperbolic_minima} мы обсуждаем относительные и гиперболические минимумы решётки, которые в параграфе \ref{sec:the_proof} используем для доказательства основного результата данной статьи.

\section{Слабая равномерная диофантова экспонента вещественного числа}\label{sec:weak_uniform_exponent_of_a_number}

\begin{definition} \label{def:weak_uniform_real_number_exponent}
  Пусть $\theta$ --- вещественное число. Супремум вещественных чисел $\gamma$, таких что для любого достаточно большого $t$ существуют целые числа $p$, $q$, удовлетворяющие неравенствам
  \[
    1\leq q\leq t,
    \qquad
    q|q\theta-p|\leq t^{1-\gamma},
  \]
  называется \emph{слабой равномерной диофантовой экспонентой} числа $\theta$ и обозначается $\hat{\hat\omega}(\theta)$.
\end{definition}

Для рациональных $\theta$ очевидно, что $\omega(\theta)=\hat{\hat\omega}(\theta)=\hat\omega(\theta)=\infty$. Для иррациональных же $\theta$ из определений и всего вышесказанного следует, что
\[
  \omega(\theta)\geq
  \hat{\hat\omega}(\theta)\geq
  \hat\omega(\theta)=1.
\]
Следующее утверждение является основным результатом статьи.

\begin{theorem}\label{t:spectrum_of_the_weak_for_a_real_number}
  Если $\theta\in\Q$, то $\hat{\hat\omega}(\theta)=\infty$. Спектр значений экспоненты $\hat{\hat\omega}(\theta)$ для иррациональных $\theta$ совпадает с отрезком $[1,2]$.
\end{theorem}

\section{Диофантовы экспоненты решёток}\label{sec:lattice_exponents}

Для каждого $\vec z=(z_1,\ldots,z_d)\in\R^d$ положим
\[
  |\vec z|=\max_{1\leq i\leq d}|z_i|,
  \qquad
  \Pi(\vec z)=\prod_{\begin{subarray}{c}1\leq i\leq d\end{subarray}}|z_i|^{1/d}.
\]
Определим также для каждого $d$-набора $\pmb\lambda=(\lambda_1,\ldots,\lambda_d)\in\R_+^d$ параллелепипед 
\begin{equation}\label{eq:prallelepipeds_lattice_exp}
  \cP(\pmb\lambda)=\Big\{\,\vec z=(z_1,\ldots,z_d)\in\R^d \ \Big|\ |z_i|\leq\lambda_i,\ i=1,\ldots,d \Big\}.
\end{equation}

\begin{definition} \label{def:regular_lattice_exponents}
  Пусть $\La$ --- решётка в $\R^d$ полного ранга. Супремум вещественных чисел $\gamma$, для которых существует сколь угодно большое $t$ и $d$-набор $\pmb\lambda\in\R_+^d$, удовлетворяющий неравенствам
  \[
    |\pmb\lambda|\leq t,
    \qquad
    \Pi(\pmb\lambda)\leq t^{-\gamma},
  \]
  такой что параллелепипед $\cP(\pmb\lambda)$ содержит ненулевые точки решётки $\La$, называется \emph{(регулярной) диофантовой экспонентой} решётки $\La$ и обозначается $\omega(\La)$.
\end{definition}

В отличие от задачи о совместных приближениях, равномерный аналог диофантовой экспоненты решётки можно определять по-разному, ибо в определении \ref{def:regular_lattice_exponents} квантор существования стоит как при $t$, так и при $\pmb\lambda$.

\begin{definition} \label{def:strong_uniform_lattice_exponents}
  Пусть $\La$ --- решётка в $\R^d$ полного ранга. Супремум вещественных чисел $\gamma$, таких что для любого достаточно большого $t$ и любого $d$-набора $\pmb\lambda\in\R_+^d$, удовлетворяющего равенствам
  \[
    |\pmb\lambda|=t,
    \qquad
    \Pi(\pmb\lambda)=t^{-\gamma},
  \]
  параллелепипед $\cP(\pmb\lambda)$ содержит ненулевые точки решётки $\La$, называется \emph{сильной равномерной диофантовой экспонентой} решётки $\La$ и обозначается $\hat\omega(\La)$.
\end{definition}

\begin{definition} \label{def:weak_uniform_lattice_exponents}
  Пусть $\La$ --- решётка в $\R^d$ полного ранга. Супремум вещественных чисел $\gamma$, таких что для любого достаточно большого $t$ существует $d$-набор $\pmb\lambda\in\R_+^d$, удовлетворяющий неравенствам
  \[
    |\pmb\lambda|\leq t,
    \qquad
    \Pi(\pmb\lambda)\leq t^{-\gamma},
  \]
  такой что параллелепипед $\cP(\pmb\lambda)$ содержит ненулевые точки решётки $\La$, называется \emph{слабой равномерной диофантовой экспонентой} решётки $\La$ и обозначается $\hat{\hat\omega}(\La)$.
\end{definition}

Из теоремы Минковского о выпуклом теле и определений диофантовых экспонент решёток легко следуют неравенства
\[
  \omega(\La)\geq
  \hat{\hat\omega}(\La)\geq
  \hat\omega(\La)\geq0.
\]

В работе \cite{german_lattice_exponents_spectrum} показано, что при $d=2$ спектр значений экспоненты $\omega(\La)$ совпадает с лучом $[0,+\infty]$ и что при $d\geq3$ он содержит отрезок
\[
  \bigg[3-\frac{d}{(d-1)^2}\,,\,+\infty\bigg].
\]
Конечно же, $0$ также принадлежит спектру при любом $d$, ибо $\omega(\La)=0$ всякий раз, когда функция $\Pi(\vec z)$ отделена от нуля в ненулевых точках решётки $\La$. К примеру, это так для любой решётки полного модуля вполне вещественного алгебраического расширения поля $\Q$ (см. \cite{borevich_shafarevich}). Теорема Шмидта о подпространствах позволяет строить решётки с регулярными диофантовыми экспонентами, принимающими значения
\begin{equation*}
  \frac{\,ab\,}{cd}\,,\qquad
  \begin{array}{l}
    a,b,c\in\N, \\
    a+b+c=d.
  \end{array}
\end{equation*}
Описание соответствующих решёток можно найти в работе \cite{german_lattice_transference}. Естественно ожидать, что при любом $d$ спектр значений экспоненты $\omega(\La)$ совпадает с лучом $[0,+\infty]$. Однако же при $d\geq3$ это до сих пор не доказано.

В работе \cite{german_comm_math_2023} показано, что сильные равномерные диофантовы экспоненты решёток принимают лишь тривиальные значения, то есть $0$ и $\infty$. А именно, если решётка $\La$ подобна подрешётке $\Z^d$ по модулю действия группы невырожденных диагональных операторов, то $\hat\omega(\La)=\infty$, а если нет, то $\hat\omega(\La)=0$.

Что касается слабых равномерных диофантовых экспонент решёток, спектр их значений нетривиален. Так, в работе \cite{german_izv_2025} доказано, что этот спектр в двумерном случае совпадает с лучом $[0,\infty]$. В многомерном случае также естественно ожидать, что спектр значений экспоненты $\hat{\hat\omega}(\La)$ совпадает с лучом $[0,+\infty]$. Однако же при $d\geq3$ это не доказано.

\section{Цепные дроби и решётки в $\R^2$}\label{sec:continued_fractions_and_lattices}

\subsection{Связь регулярной диофантовой экспоненты вещественного числа с ростом его неполных частных}

Для каждого $\theta\in\R$ в соответствии с определением \ref{def:belpha_one_number} экспонента $\omega(\theta)$ равна супремуму таких вещественных $\gamma$, что неравенство
\[
  |\theta x-y|\leqslant|x|^{-\gamma}
\]
имеет бесконечно много решений в ненулевых целых $x$, $y$.

Благодаря классическому соотношению
\begin{equation}\label{eq:omega_of_a_number_vs_its_continued_fraction}
  \omega(\theta)=1+\limsup_{k\to+\infty}\frac{\ln a_{k+1}}{\ln q_k},
\end{equation}
связывающему $\omega(\theta)$ с цепной дробью числа $\theta$,
\[
  \theta=[a_0;a_1,a_2,\ldots],
  \qquad
  \frac{p_k}{q_k}=[a_0;a_1,\ldots,a_k],
\]
можно явно строить числа с любой заданной наперёд диофантовой экспонентой $\omega(\theta)\geq1$. Действительно, если для заданного $\gamma\geq0$ положить $a_0=1$ и для каждого $k\geq0$ положить
\[
  a_{k+1}=[q_k^\gamma]+1,
\]
где $[\,\cdot\,]$ обозначает целую часть, то из \eqref{eq:omega_of_a_number_vs_its_continued_fraction} очевидно, что для такого $\theta$ справедливо $\omega(\theta)=1+\gamma$.

\subsection{Регулярные диофантовы экспоненты вещественных чисел и решёток}

Покажем, как связаны регулярные диофантовы экспоненты решёток в $\R^2$ с регулярными диофантовыми экспонентами вещественных чисел.

Для любой решётки $\La$ полного ранга в соответствии с определением \ref{def:regular_lattice_exponents} экспонента $\omega(\La)$ равна супремуму таких вещественных $\gamma$, что неравенство
\[
  \Pi(\vec z)\leq|\vec z|^{-\gamma}
\]
имеет бесконечно много решений в $\vec z\in\La$.

Пусть $\theta$, $\eta$ --- различные вещественные числа, $\theta>1$, $\eta>1$. Рассмотрим линейные формы $L_1$, $L_2$ от двух переменных с коэффициентами, записанными в строчках матрицы
\[
  A=
  \begin{pmatrix}
    \theta & -1 \\
    1 & \phantom{-}\eta
  \end{pmatrix},
\]
и положим
\[
  \La=A\Z^2=
  \Big\{\big(L_1(\vec u),L_2(\vec u)\big)\,\Big|\,\vec u\in\Z^2 \Big\}.
\]
Тогда для каждого $\,\vec z=\big(L_1(\vec u),L_2(\vec u)\big)\in\La$, где $\ \vec u=(x,y)\in\Z^2$, справедливо
\[
  \Pi(\vec z)^2=|L_1(\vec u)|\cdot|L_2(\vec u)|=|\theta x-y|\cdot|x+\eta y|,
\]
а также
\[
  |\vec z|\asymp|L_i(\vec u)|\asymp|\vec u|\asymp|x|\asymp|y|
  \ \text{ при }\ |L_j(\vec u)|\leq1,\ i\neq j,
\]
где константы, подразумеваемые символом ``$\asymp$'', зависят лишь от $\theta$ и $\eta$. Следовательно, при $|x+\eta y|\leq1$ справедливо
\[
  \big(\Pi(\vec z)\cdot|\vec z|^\gamma\big)^2
  \asymp
  |\theta x-y|\cdot|x|^{1+2\gamma},
\]
тогда как при $|\theta x-y|\leq1$ справедливо
\[
  \big(\Pi(\vec z)\cdot|\vec z|^\gamma\big)^2
  \asymp
  |x+\eta y|\cdot|y|^{1+2\gamma}.
\]
Отсюда видим, что 
\begin{equation}\label{eq:omega_La_vs_the_maximum}
  \omega(\La)=\frac12\Big(\max\big(\omega(\theta),\omega(\eta)\big)-1\Big).
\end{equation}

При $\omega(\eta)=\omega(\theta)$ пропадает необходимость брать максимум. Особенно удобным для некоторых приложений является выбор $\eta=\theta$. Тогда для решётки
\begin{equation}\label{eq:La_theta}
  \La_\theta=
  \begin{pmatrix}
    \theta & -1 \\
    1 & \phantom{-}\theta
  \end{pmatrix}
  \Z^2
\end{equation}
соотношение \eqref{eq:omega_La_vs_the_maximum} сводится к равенству
\begin{equation}\label{eq:omega_La_vs_no_maximum}
  \omega(\La_\theta)=
  \frac{\omega(\theta)-1}2\,,
\end{equation}
из которого очевидным образом следует, что спектр значений экспоненты $\omega(\La)$ совпадает с отрезком $[0,\infty]$.

\begin{remark*}
  Решётка $\La_\theta$ переходит в себя при повороте $\R^2$ вокруг начала координат на $90^\circ$. Данная особенность решётки $\La_\theta$ является весьма удобной для различных приложений.
\end{remark*}

\subsection{Слабые равномерные диофантовы экспоненты вещественных чисел и решёток}

Для любой решётки $\La$ полного ранга в соответствии с определением \ref{def:weak_uniform_lattice_exponents} экспонента $\hat{\hat\omega}(\La)$ равна супремуму таких вещественных $\gamma$, что для каждого достаточно большого $t$ система неравенств
\begin{equation}\label{eq:weak_uniform_lattice_exponents_system}
  |\vec z|\leq t,
  \qquad
  \Pi(\vec z)\leq t^{-\gamma}
\end{equation}
имеет решение в ненулевых $\vec z\in\La$.


Пусть $\theta>1$. Покажем, как связаны экспоненты $\hat{\hat\omega}(\theta)$ и $\hat{\hat\omega}(\La_\theta)$, где $\La_\theta$ определяется равенством \eqref{eq:La_theta}. Поскольку при $|q\theta-p|\leq1$
\[
  q+\theta p\asymp q,
\]
где константы, подразумеваемые символом ``$\asymp$'', зависят лишь от $\theta$, экспонента $\hat{\hat\omega}(\theta)$ в соответствии с определением \ref{def:weak_uniform_real_number_exponent} равна супремуму вещественных чисел $\gamma$, таких что для любого достаточно большого $t$ существуют не равные одновременно нулю целые числа $p$, $q$, удовлетворяющие неравенствам
\[
  |q+\theta p|\leq t,
  \qquad
  |(q+\theta p)(q\theta-p)|\leq t^{1-\gamma}.
\]
Учитывая же инвариантность решётки $\La_\theta$ относительно поворота $\R^2$ вокруг начала координат на $90^\circ$, получаем, что $\hat{\hat\omega}(\theta)$ равна супремуму таких вещественных $\gamma$, что для каждого достаточно большого $t$ система неравенств
\begin{equation}\label{eq:weak_uniform_exponents_for_a_real_number_system}
  |\vec z|\leq t,
  \qquad
  \Pi(\vec z)\leq t^{(1-\gamma)/2}
\end{equation}
имеет решение в ненулевых $\vec z\in\La_\theta$. Стало быть, имеет место аналог равенства \eqref{eq:omega_La_vs_no_maximum}:
\begin{equation}\label{eq:hat_hat_omega_La_theta_vs_hat_hat_omega_theta}
  \hat{\hat\omega}(\La_\theta)=
  \frac{\hat{\hat\omega}(\theta)-1}2\,.
\end{equation}

Ввиду соотношения \eqref{eq:hat_hat_omega_La_theta_vs_hat_hat_omega_theta} теорема \ref{t:spectrum_of_the_weak_for_a_real_number} равносильна следующему утверждению.

\begin{theorem}\label{t:spectrum_of_the_weak_for_lattices}
  Если $\theta\in\Q$, то $\hat{\hat\omega}(\La_\theta)=\infty$. Спектр значений экспоненты $\hat{\hat\omega}(\La_\theta)$ для иррациональных $\theta$ совпадает с отрезком $[0,1/2]$.
\end{theorem}

\begin{remark*}
  Как было сказано выше, спектр слабых равномерных диофантовых экспонент решёток в двумерном случае совпадает с лучом $[0,\infty]$. По утверждению теоремы \ref{t:spectrum_of_the_weak_for_lattices} легко заметить, что решёток вида \eqref{eq:La_theta} недостаточно для описания полного спектра значений экспоненты $\hat{\hat\omega}(\La)$. Как показано в работе \cite{german_izv_2025}, для изучения полного спектра экспоненты $\hat{\hat\omega}(\La)$ нужно рассматривать пары различных чисел $\theta$, $\eta$, диофантовы свойства которых связаны определённым образом друг с другом.
\end{remark*}

\section{Относительные и гиперболические минимумы}\label{sec:relative_and_hyperbolic_minima}

\begin{definition}
  Пусть $\La$ --- решётка полного ранга в $\R^d$. Точка $\vec z=(z_1,\ldots,z_d)\in\La$ называется \emph{относительным минимумом} решётки $\La$, если не существует таких ненулевых точек $\vec w=(w_1,\ldots,w_d)\in\La$, что
  \[
    |w_i|\leq|z_i|,\quad i=1,\ldots,d,
    \qquad\text{ и }\qquad
    \sum_{i=1}^d|w_i|<\sum_{i=1}^d|z_i|.
  \] 
\end{definition}

Иными словами, точка $\vec z\in\La$ является относительным минимумом решётки $\La$, если в параллелепипеде $\cP(\vec z)$, помимо вершин и точки начала координат, точек решётки нет.

\begin{definition}\label{def:hyperbolic_minimum}
  Пусть $\La$ --- решётка полного ранга в $\R^d$. Будем называть точку $\vec z\in\La$ \emph{гиперболическим минимумом} решётки $\La$, если не существует таких ненулевых точек $\vec w\in\La$, что 
  \[
    |\vec w|\leq|\vec z|
    \qquad\text{ и }\qquad
    \Pi(\vec w)<\Pi(\vec z).
  \] 
\end{definition}

Положим
\[
  \cH(\vec z)=\Big\{\vec w\in\R^d \,\Big|\, |\vec w|\leq|\vec z|,\ \Pi(\vec w)<\Pi(\vec z) \Big\}.
\]
Из определения \ref{def:hyperbolic_minimum} очевидно, что ненулевая точка $\vec z$ решётки $\La$ является \emph{гиперболическим минимумом} этой решётки тогда и только тогда, когда $\cH(\vec z)\cap\La=\{\vec 0\}$.

\begin{lemma}\label{l:hyperbolic_is_also_relative}
  Каждый гиперболический минимум решётки $\La$ является и относительным минимумом этой решётки.
\end{lemma}

\begin{proof}
  Если $\vec z$ --- гиперболический минимум решётки $\La$, ненулевые точки этой решётки, содержащиеся в $\cH(\vec z)$, обязаны располагаться на пересечении границы $\cH(\vec z)$ с поверхностью $\big\{ \vec w\in\R^d \,\big|\, \Pi(\vec w)=\Pi(\vec z) \big\}$. В частности, в этом случае любая ненулевая точка решётки $\La$, содержащаяся в $\cP(\vec z)$, обязана быть вершиной $\cP(\vec z)$. Таким образом, каждый гиперболический минимум решётки $\La$, действительно, является и относительным минимумом этой решётки.
\end{proof}

\subsection{Последовательные гиперболические минимумы}

Если $\vec z$ --- гиперболический минимум решётки $\La$, то, конечно же, точка $-\vec z$ также является гиперболическим минимумом этой решётки. Может так оказаться, что существуют другие гиперболические минимумы с такими же значениями функционалов $|\cdot|$ и $\Pi(\,\cdot\,)$. Выберем из каждого такого набора по одному представителю и упорядочим их по возрастанию $|\cdot|$. Получим последовательность $\vec z_1,\vec z_2,\vec z_3,\ldots$ гиперболических минимумов (см. рис. \ref{fig:successive_hyperbolic_minima}), причём $|\vec z_k|<|\vec z_{k+1}|$ и $\hat\cH(\vec z_k)\cap\La=\{\vec 0\}$, где
\[
  \hat\cH(\vec z_k)=\Big\{\vec w\in\R^d \,\Big|\, |\vec w|<|\vec z_{k+1}|,\ \Pi(\vec w)<\Pi(\vec z_k) \Big\}.
\]

\begin{figure}[h]
\centering
\begin{tikzpicture}[domain=-5.1:5.1,scale=1.3]

  \draw (-5.3,0) -- (5.3,0); 
  \draw (0,-5.3) -- (0,5.3); 

  \fill[fill=blue,opacity=0.15]
      plot [domain=0.2:5] (\x,{1/(\x)}) --
      plot [domain=5:0.2] (\x, {-1/(\x)}) --
      plot [domain=-0.2:-5] (\x,{1/(\x)}) --
      plot [domain=-5:-0.2] (\x, {-1/(\x)}) --
      cycle;

  \fill[fill=red,opacity=0.15]
      plot [domain=0.75:2] (\x,{1.5/(\x)}) --
      plot [domain=2:0.75] (\x, {-1.5/(\x)}) --
      plot [domain=-0.75:-2] (\x,{1.5/(\x)}) --
      plot [domain=-2:-0.75] (\x, {-1.5/(\x)}) --
      cycle;

  \draw[color=blue]
      plot [domain=0.2:5] (\x,{1/(\x)}) --
      plot [domain=5:0.2] (\x, {-1/(\x)}) --
      plot [domain=-0.2:-5] (\x,{1/(\x)}) --
      plot [domain=-5:-0.2] (\x, {-1/(\x)}) --
      cycle;

  \draw[color=red]
      plot [domain=0.75:2] (\x,{1.5/(\x)}) --
      plot [domain=2:0.75] (\x, {-1.5/(\x)}) --
      plot [domain=-0.75:-2] (\x,{1.5/(\x)}) --
      plot [domain=-2:-0.75] (\x, {-1.5/(\x)}) --
      cycle;

  \node[fill=black,circle,inner sep=1.3pt] at (0,0) {};
  \node[fill=black,circle,inner sep=1.3pt] at (1.03,1.5/1.03) {};
  \node[fill=black,circle,inner sep=1.3pt] at (-1.03,-1.5/1.03) {};
  \node[fill=black,circle,inner sep=1.3pt] at (-0.5,2) {};
  \node[fill=black,circle,inner sep=1.3pt] at (0.5,-2) {};
  \node[fill=black,circle,inner sep=1.3pt] at (5,0.1) {};
  \node[fill=black,circle,inner sep=1.3pt] at (-5,-0.1) {};
  
  \draw(1.1,1.5/1.1+0.1) node[right]{$\vec z_{k-1}$};
  \draw(-1.03,-1.5/1.1-0.1) node[left]{$-\vec z_{k-1}$};
  \draw(0.47,-2.23) node[right]{$\vec z_k$};
  \draw(-0.47,2.23) node[left]{$-\vec z_k$};
  \draw(5,0.2) node[right]{$\vec z_{k+1}$};
  \draw(-4.93,-0.2) node[left]{$-\vec z_{k+1}$};

  \draw(4,0.3) node[above]{${\color[rgb]{0,0,1}\hat\cH(\vec z_k)}$};
  \draw(1.3,-1.35) node[right]{${\color[rgb]{1,0,0}\hat\cH(\vec z_{k-1})}$};

\end{tikzpicture}
\caption{Последовательные гиперболические минимумы}\label{fig:successive_hyperbolic_minima}
\end{figure}

Если функционал $\Pi(\,\cdot\,)$ принимает сколь угодно малые значения на ненулевых точках решётки $\La$, но при этом ни на какой такой точке не обращается в ноль, гиперболических минимумов бесконечно много. В этом случае для экспонент $\omega(\La)$, $\hat{\hat\omega}(\La)$ справедливы равенства
\[
  \omega(\La)=\sup\Big\{ \gamma\in\R \,\Big|\, \forall\,K\in\N\,\ \exists\,k\geq K:\,\Pi(\vec z_k)\leq|\vec z_k|^{-\gamma} \Big\},\ \
\]
\begin{equation}\label{eq:hat_hat_omega_in_terms_of_hyperbolic_minima}
  \hat{\hat\omega}(\La)=\sup\Big\{ \gamma\in\R \,\Big|\, \exists\,K\in\N:\,\forall k\geq K\ \Pi(\vec z_k)\leq|\vec z_{k+1}|^{-\gamma} \Big\}.
\end{equation}

\subsection{Относительные и гиперболические минимумы для решёток вида $\La_\theta$}

Относительный минимум решётки, вообще говоря, не обязан быть гиперболическим минимумом. Тем не менее, при определённых условиях удаётся доказать, что относительные минимумы решётки являются также гиперболическими минимумами. Например, в двумерном случае для решёток вида $\La_\theta$ с некоторыми условиями на рост неполных частных числа $\theta$.

Пусть далее $d=2$. Тогда, как показано, например, в работе \cite{german_mathnotes_2006}, множество относительных минимумов решётки $\La_\theta=A\Z^2$, где
\begin{equation}\label{eq:A_theta}
  A=
  \begin{pmatrix}
    \theta & -1 \\
    1 & \phantom{-}\theta
  \end{pmatrix},
  \qquad
  \theta>1,
\end{equation}
совпадает с множеством вершин её четырёх полигонов Клейна
\[
  \cK_{\pmb\e}(\La_\theta)=
  \conv\Big(\Big\{ \vec z=(z_1,z_2)\in\La_\theta\backslash\{\vec 0\} \,\Big|\, \e_iz_i\geq0,\ i=1,2 \Big\}\Big),
\]
\[
  \pmb\e=(\e_1,\e_2),
  \quad
  \e_1,\e_2=\pm1.
\]
Вершины же этих полигонов Клейна соответствуют подходящим дробям числа $\theta$. А именно, их вершины являются образами при действии оператора $A$ точек
\begin{equation}\label{eq:vertices}
  \pm(1,0),\quad
  \pm(0,1),\quad
  \pm(q_k,p_k),\quad
  \pm(p_k,-q_k),\quad
  k=0,1,2,\ldots,
\end{equation}
где $p_k/q_k$ --- подходящие дроби числа $\theta$. Подробное описание данного соответствия можно найти в работе \cite{german_tlyustangelov_MJCNT_2016}, а также в обзоре \cite{german_UMN_2023}.

В частности, ввиду леммы \ref{l:hyperbolic_is_also_relative} все гиперболические минимумы решётки $\La$ являются образами при действии оператора $A$ каких-то точек вида \eqref{eq:vertices}.

\begin{lemma}\label{l:relative_is_hyperbolic}
  Пусть $\theta\in\R\backslash\Q$, $1<\theta<2$. Пусть для числителей и знаменателей подходящих дробей числа $\theta$ при всех достаточно больших $k$ справедливы неравенства
  \[
    q_k|q_k\theta-p_k|
    >3
    q_{k+1}|q_{k+1}\theta-p_{k+1}|.
  \]
  Пусть $\La_\theta=A\Z^2$, где $A$ определяется равенством \eqref{eq:A_theta}. Тогда любой относительный минимум $\vec z$ решётки $\La_\theta$ с достаточно большим $|\vec z|$ является и гиперболическим минимумом этой решётки.
\end{lemma}

\begin{proof}
  Поскольку $1<\theta<2$ и $1\leq p_k/q_k\leq2$ для каждого $k$, справедливы неравенства
  \[
    2q_k<q_k+p_k\theta<5q_k.
  \]
  Стало быть, существует такое $K$, что для всех $k>K$ выполняется
  \begin{multline}\label{eq:product_decaying}
    \big|(q_k+p_k\theta)(q_k\theta-p_k)\big|>
    2q_k|q_k\theta-p_k|
    > \\ >
    6q_{k+1}|q_{k+1}\theta-p_{k+1}|>
    \frac65\,
    \big|(q_{k+1}+p_{k+1}\theta)(q_{k+1}\theta-p_{k+1})\big|.
  \end{multline}
  Пусть $\vec z=(z_1,z_2)$ --- относительный минимум решётки $\La_\theta$ с $|z_1|\leq1$. Тогда существует такой индекс $k$, что $z_1=q_k\theta-p_k$ и $z_2=q_k+p_k\theta$. Ввиду \eqref{eq:product_decaying} найдётся такое $m>K$, что
  \[
    \big|(q_m\theta-p_m)(q_m+p_m\theta)\big|=
    \min_{0\leq k\leq m}
    \big|(q_k\theta-p_k)(q_k+p_k\theta)\big|.
  \]
  Учитывая, что решётка $\La_\theta$ переходит в себя при повороте $\R^2$ вокруг начала координат на $90^\circ$, получаем, что точка $\vec v_m=(q_m\theta-p_m,q_m+p_m\theta)$ является гиперболическим минимумом решётки $\La_\theta$.
  
  Покажем, что точка $\vec v_{m+1}=(q_{m+1}\theta-p_{m+1},q_{m+1}+p_{m+1}\theta)$ также является гиперболическим минимумом. Рассмотрим прямоугольник
  \[
    \cP=\Big\{(z_1,z_2)\in\R^2 \,\Big|\, |z_1|\leq|q_m\theta-p_m|,\ |z_2|\leq q_{m+1}+p_{m+1}\theta \Big\}
  \]
  \big(то есть $\cP=\cP(\pmb\lambda)$, где $\pmb\lambda=(|q_m\theta-p_m|,q_{m+1}+p_{m+1}\theta)$\big) и параллелограмм
  \[
    \cQ=\Big\{(x,y)\in\R^2 \,\Big|\, |x\theta-y|\leq|q_m\theta-p_m|,\ |x|\leq q_{m+1} \Big\}.
  \]
  
  \begin{figure}[h]
\centering
\begin{tikzpicture}[domain=-5.1:5.1,scale=1.3]

  \draw[->,>=stealth'] (-5.5,0) -- (5.5,0) node[right] {$z_2$};
  \draw[->,>=stealth'] (0,-5.5) -- (0,5.5) node[above] {$z_1$};

  \fill[fill=blue,opacity=0.15]
      plot [domain=0.2:5] (\x,{1/(\x)}) --
      plot [domain=5:0.2] (\x, {-1/(\x)}) --
      plot [domain=-0.2:-5] (\x,{1/(\x)}) --
      plot [domain=-5:-0.2] (\x, {-1/(\x)}) --
      cycle;

  \fill[fill=red,opacity=0.15]
      plot [domain=0.75:2] (\x,{1.5/(\x)}) --
      plot [domain=2:0.75] (\x, {-1.5/(\x)}) --
      plot [domain=-0.75:-2] (\x,{1.5/(\x)}) --
      plot [domain=-2:-0.75] (\x, {-1.5/(\x)}) --
      cycle;

  \fill[fill=teal,opacity=0.1]
      (-5,0.75) -- (5,0.75) -- (5,-0.75) -- (-5,-0.75) -- cycle;

  \fill[fill=teal,opacity=0.1]
      (0.75,-5) -- (0.75,5) -- (-0.75,5) -- (-0.75,-5) -- cycle;

  \draw[color=blue]
      plot [domain=0.2:5] (\x,{1/(\x)}) --
      plot [domain=5:0.2] (\x, {-1/(\x)}) --
      plot [domain=-0.2:-5] (\x,{1/(\x)}) --
      plot [domain=-5:-0.2] (\x, {-1/(\x)}) --
      cycle;

  \draw[color=red]
      plot [domain=0.75:2] (\x,{1.5/(\x)}) --
      plot [domain=2:0.75] (\x, {-1.5/(\x)}) --
      plot [domain=-0.75:-2] (\x,{1.5/(\x)}) --
      plot [domain=-2:-0.75] (\x, {-1.5/(\x)}) --
      cycle;
  
  \draw[color=teal] (-5,0.75) -- (5,0.75) -- (5,-0.75) -- (-5,-0.75) -- cycle;

  \draw[color=teal] (0.75,-5) -- (0.75,5) -- (-0.75,5) -- (-0.75,-5) -- cycle;

  \node[fill=black,circle,inner sep=1.3pt] at (0,0) {};
  \node[fill=black,circle,inner sep=1.3pt] at (2,0.75) {};
  \node[fill=black,circle,inner sep=1.3pt] at (-2,-0.75) {};
  \node[fill=black,circle,inner sep=1.3pt] at (-0.75,2) {};
  \node[fill=black,circle,inner sep=1.3pt] at (0.75,-2) {};
  \node[fill=black,circle,inner sep=1.3pt] at (1/5,5) {};
  \node[fill=black,circle,inner sep=1.3pt] at (-1/5,-5) {};
  \node[fill=black,circle,inner sep=1.3pt] at (5,-1/5) {};
  \node[fill=black,circle,inner sep=1.3pt] at (-5,1/5) {};
  
  \draw(1.95,0.9) node[right]{$\vec v_m$};
  \draw(-1.9,-0.9) node[left]{$-\vec v_m$};
  \draw(-0.72,1.94) node[above left]{$R(\vec v_m)$};
  \draw(0.73,-1.9) node[below right]{$-R(\vec v_m)$};
  \draw(0.13,5.23) node[right]{$R(\vec v_{m+1})$};
  \draw(-0.1,-5.25) node[left]{$-R(\vec v_{m+1})$};
  \draw(5,-0.2) node[below right]{$\vec v_{m+1}$};
  \draw(-4.97,0.14) node[above left]{$-\vec v_{m+1}$};

  \draw(4,0.22) node[above,color=blue]{$\cH(\vec v_{m+1})$};
  \draw(1.27,-1.35) node[right,color=red]{$\cH(\vec v_m)$};
  \draw(-4,0.75) node[above,color=teal]{$\cP$};
  \draw(0.75,4) node[right,color=teal]{$R(\cP)$};

\end{tikzpicture}
\caption{Покрытие $\cH(\vec v_{m+1})$}\label{fig:H_v_m+1_covering}
\end{figure}

  Поскольку $|q_m\theta-p_m|<1/2$ при достаточно большом $m$, а оператор $A$ является композицией поворота и гомотетии с коэффициентом $\sqrt{\theta^2+1}$, справедливо включение
  \[
    \cP\subset A(\cQ)\cup(\vec v_{m+1}+\cB)\cup(-\vec v_{m+1}+\cB),
  \]
  где $\cB$ обозначает открытый круг радиуса $\sqrt{\theta^2+1}$ с центром в начале координат. В этом круге ненулевых точек решётки $\La_\theta$ нет. Следовательно, 
  \[
    \cP\cap\La_\theta=\{\vec 0,\pm\vec v_m,\pm\vec v_{m+1}\},
  \]
  причём во внутренности $\cP$ ненулевых точек решётки нет. Учитывая \eqref{eq:product_decaying}, получаем, что
  \[
    \cH(\vec v_{m+1})\cap\La_\theta=\{\vec 0,\pm\vec v_{m+1},R(\pm\vec v_{m+1})\},
  \]
  где $R$ обозначает поворот $\R^2$ вокруг начала координат на $90^\circ$, ибо справедливо включение (см. рис. \ref{fig:H_v_m+1_covering})
  \[
    \cH(\vec v_{m+1})
    \subset
    \cH(\vec v_m)\cup\cP\cup R(\cP).
  \]
  Стало быть, точка $\vec v_{m+1}$, действительно, является гиперболическим минимумом решётки $\La_\theta$.

  Повторяя данное рассуждение, получаем, что все точки $\vec v_k=(q_k\theta-p_k,q_k+p_k\theta)$ с $k\geq m$ являются гиперболическими минимумами.
\end{proof}

\begin{corollary*}
  Пусть $\theta$ и $\La_\theta$ удовлетворяют условию леммы \ref{l:relative_is_hyperbolic}. Положим
  \[
    \vec v_k=(q_k\theta-p_k,q_k+p_k\theta)
  \]
  для каждого $k\geq0$. Тогда 
  \begin{equation}\label{eq:hat_hat_omega_in_terms_of_relative_minima}
    \hat{\hat\omega}(\La_\theta)=\sup\Big\{ \gamma\in\R \,\Big|\, \exists\,K\in\N:\,\forall k\geq K\ \Pi(\vec v_k)\leq|\vec v_{k+1}|^{-\gamma} \Big\}.
  \end{equation} 
\end{corollary*}

\begin{proof}
  Достаточно воспользоваться леммой \ref{l:relative_is_hyperbolic} и равенством \eqref{eq:hat_hat_omega_in_terms_of_hyperbolic_minima}.
\end{proof}

\subsection{Ограниченность $\hat{\hat\omega}(\La_\theta)$}

\begin{lemma}\label{l:the_next_hyperbolic_minimum_is_large}
  Пусть $\La_\theta=A\Z^2$, где $A$ определяется равенством \eqref{eq:A_theta}. Пусть $(\vec z_k)_{k\in\N}$ --- последовательность её гиперболических минимумов. Тогда для любого $k\geq2$ справедливо
  \begin{equation}\label{eq:the_next_hyperbolic_minimum_is_large}
    |\vec z_{k+1}|>\frac43\varphi^{k-2}\Pi(\vec z_k)^{-2},
  \end{equation}
  где $\varphi=\frac{1+\sqrt5}{2}$.
\end{lemma}

\begin{proof}
  Поскольку решётка $\La_\theta$ инвариантна относительно поворота на $90^\circ$ вокруг начала координат, можно считать, что $\vec z_k$ выбраны таким образом, что первая координата каждой точки $\vec z_k$ ограничена единицей. По лемме \ref{l:hyperbolic_is_also_relative} все точки $\vec z_k$ являются относительными минимумами решётки $\La_\theta$. Стало быть, для каждого $k$ найдётся такое $n=n(k)$, что $\vec z_k=\vec v_n=(q_n\theta-p_n,q_n+p_n\theta)$ и для каждого $k$ справедливо $n(k+1)>n(k)\geq k-1$. Здесь, как и прежде, $p_n/q_n=[a_0;a_1,a_2,\ldots,a_n]$ --- подходящая дробь числа $\theta$. Для неполного частного $a_{n+1}$ и знаменателей $q_{n-1},q_n,q_{n+1}$ справедливо равенство $q_{n+1}=a_{n+1}q_n+q_{n-1}$. Известно также, что $q_n>\varphi^{n-1}$ (данное неравенство легко доказывается по индукции). Стало быть, поскольку $\theta>1$, для $n=n(k)$ справедливо
  \begin{equation}\label{eq:the_next_hyperbolic_minimum_is_large_first_half}
    a_{n+1}<
    \frac{q_{n+1}}{q_n}<
    \frac{q_{n+1}}{\varphi^{n-1}}<
    \frac{q_{n+1}+p_{n+1}\theta}{2\varphi^{n-1}}=
    \frac{|\vec v_{n+1}|}{2\varphi^{n-1}}\leq
    \frac{|\vec z_{k+1}|}{2\varphi^{k-2}}.
%
  \end{equation}
  При этом, как известно (см., например, \cite{lang,khintchine_CF,schmidt_DA}),
  \begin{equation}\label{eq:product_vs_partial_quotient}
    \frac1{a_{n+1}+2}<
    q_n|q_n\theta-p_n|<
    \frac1{a_{n+1}}\,.
  \end{equation}
  Следовательно, 
  \begin{multline}\label{eq:the_next_hyperbolic_minimum_is_large_second_half}
    a_{n+1}^{-1}\leq
    3(a_{n+1}+2)^{-1}<
    3q_n|q_n\theta-p_n|< \\ <
    \frac32\big|(q_n\theta-p_n)(q_n+p_n\theta)\big|=
    \frac32\Pi(\vec v_n)^2=
    \frac32\Pi(\vec z_k)^2.
  \end{multline}
  Из \eqref{eq:the_next_hyperbolic_minimum_is_large_first_half} и \eqref{eq:the_next_hyperbolic_minimum_is_large_second_half} получаем
  \[
    |\vec z_{k+1}|>
    2\varphi^{k-2}a_{n+1}>
    \frac43\varphi^{k-2}\Pi(\vec z_k)^{-2}.
  \]
\end{proof}

\begin{corollary*}
  Пусть $\theta\in\R\backslash\Q$, $\theta>1$. Тогда справедливы оценки
  \begin{equation}\label{eq:upper_bound_for_hat_hat_omega}
    \hat{\hat\omega}(\La_\theta)\leq\frac12\,,
    \qquad
    \hat{\hat\omega}(\theta)\leq2.
  \end{equation}
\end{corollary*}

\begin{proof}
  Из неравенства \eqref{eq:the_next_hyperbolic_minimum_is_large} следует, что $\Pi(\vec z_k)>|\vec z_{k+1}|^{-1/2}$.
  Остаётся воспользоваться соотношениями \eqref{eq:hat_hat_omega_in_terms_of_hyperbolic_minima} и \eqref{eq:hat_hat_omega_La_theta_vs_hat_hat_omega_theta}.
\end{proof}

\section{Доказательство теоремы \ref{t:spectrum_of_the_weak_for_lattices}}\label{sec:the_proof}

Зафиксируем $\gamma>0$. Положим $a_0=1$. Для каждого целого $k\geq0$ положим
\begin{equation}\label{eq:next_partial_qoutient_definition}
  a_{k+1}=[q_k^\gamma]+1,
\end{equation}
где $[\,\cdot\,]$ обозначает целую часть, а $q_k$ --- знаменатель дроби $p_k/q_k=[a_0;a_1,a_2,\ldots,a_k]$. Рассмотрим число $\theta=[a_0;a_1,a_2,\ldots]$. Как мы отмечали выше (см. \eqref{eq:product_vs_partial_quotient}),
\[
  \frac1{a_{k+1}+2}<
  q_k|q_k\theta-p_k|<
  \frac1{a_{k+1}}\,.
\]
Учитывая \eqref{eq:next_partial_qoutient_definition}, получаем
\begin{equation}\label{eq:bounds_for_the_product}
  \frac1{q_k^\gamma+3}<
  q_k|q_k\theta-p_k|<
  \frac1{q_k^\gamma}\,.
\end{equation}
Из \eqref{eq:next_partial_qoutient_definition} также следует, что
\begin{equation}\label{eq:consecutive_denominators_inequality}
  q_{k+1}=
  a_{k+1}q_k+q_{k-1}>
  q_k^{1+\gamma}
\end{equation}
и
\begin{equation}\label{eq:consecutive_denominators_asymptotics}
  q_{k+1}<
  (q_k^\gamma+1)q_k+q_k<
  3q_k^{1+\gamma}
\end{equation}
Заметим, что при $q_k>12^{1/{\gamma^2}}$ выполняется $q_k^{\gamma^2}>12\geq3+9q_k^{-\gamma}$, откуда
\begin{equation}\label{eq:technical_inequality_for_q_k}
  \frac1{q_k^\gamma+3}>
  \frac3{q_k^{\gamma(1+\gamma)}}\,.
\end{equation}
Собирая вместе \eqref{eq:bounds_for_the_product}, \eqref{eq:consecutive_denominators_inequality}, \eqref{eq:technical_inequality_for_q_k}, получаем, что при фиксированном $\gamma$ для всех достаточно больших $k$ справедливо
\[
  q_k|q_k\theta-p_k|>
  \frac1{q_k^\gamma+3}>
  \frac3{q_k^{\gamma(1+\gamma)}}>
  \frac3{q_{k+1}^\gamma}>
  3q_{k+1}|q_{k+1}\theta-p_{k+1}|,
\]
то есть условие леммы \ref{l:relative_is_hyperbolic} выполняется. Стало быть, если обозначить
\[
  \vec v_k=(q_k\theta-p_k,q_k+p_k\theta)
\]
и заметить, что ввиду \eqref{eq:bounds_for_the_product}, \eqref{eq:consecutive_denominators_inequality}, \eqref{eq:consecutive_denominators_asymptotics} справедливы оценки
\begin{multline*}
  \Pi(\vec v_k)=
  \big|(q_k\theta-p_k)(q_k+p_k\theta)\big|^{1/2}
  \asymp
  \big(q_k|q_k\theta-p_k|\big)^{1/2}
  \asymp \\ \asymp
  q_k^{-\gamma/2}
  \asymp
  q_{k+1}^{-\gamma/(2+2\gamma)}
  \asymp
  (q_{k+1}+p_{k+1}\theta)^{-\gamma/(2+2\gamma)}=
  |\vec v_{k+1}|^{-\gamma/(2+2\gamma)},
\end{multline*}
получим благодаря соотношению \eqref{eq:hat_hat_omega_in_terms_of_relative_minima} равенство
\[
  \hat{\hat\omega}(\Lambda_\theta)=\gamma/(2+2\gamma).
\]

Образом множества положительных чисел при отображении $\gamma\mapsto\gamma/(2+2\gamma)$ является интервал $(0,1/2)$. Стало быть, весь этот интервал принадлежит спектру значений $\hat{\hat\omega}(\Lambda)$.

Далее, если в качестве $\theta$ взять иррациональное число с ограниченными неполными частными, то для такого числа $\omega(\theta)=1$ и в соответствии с \eqref{eq:omega_La_vs_no_maximum} имеем $\omega(\La_\theta)=0$. Поскольку же $0\leq\hat{\hat\omega}(\La_\theta)\leq\omega(\La_\theta)$, получаем $\hat{\hat\omega}(\La_\theta)=0$.

Если же вместо роста неполных частных, задаваемого \eqref{eq:next_partial_qoutient_definition}, рассмотреть более быстрый рост --- к примеру, определить неполные частные соотношениями
\begin{equation}\label{eq:next_partial_qoutient_definition_exponential}
  a_{k+1}=[q_k^k]+1,
\end{equation}
то, следуя шагам рассуждения, приведённого выше, несложно показать, что для такого $\theta$ будет выполнено $\hat{\hat\omega}(\La_\theta)=1/2$.

Значения $\hat{\hat\omega}(\La_\theta)$ большие $1/2$ ввиду \eqref{eq:upper_bound_for_hat_hat_omega} при иррациональном $\theta$ не достигаются. Если же в качестве $\theta$ взять число рациональное, то очевидно, что тогда $\hat{\hat\omega}(\La_\theta)=\omega(\La_\theta)=\infty$.

Теорема доказана.

\section*{Благодарности}
Автор является победителем конкурса <<Ведущий учёный>> Фонда развития теоретической физики и математики <<БАЗИС>> и хотел бы поблагодарить спонсоров и жюри конкурса.

\end{document}